\documentclass{amsart}

\usepackage{color}
\usepackage{amsmath,amssymb}
\usepackage[latin2]{inputenc}
\usepackage{t1enc}
\usepackage{bbm}

\newtheorem{theorem}{Theorem}[section]
\newtheorem{lem}[theorem]{Lemma}

\theoremstyle{definition}

\theoremstyle{remark}

\numberwithin{equation}{section}

\begin{document}

\title{Fractional order elliptic problems with inhomogeneous Dirichlet boundary conditions}


\author{Ferenc Izs\'ak}
\address{Department of Applied Analysis and Computational Mathematics $\&$
 MTA ELTE NumNet Research Group,
E\"otv\"os Lor\'and University, P\'azm\'any P. stny. 1C, 1117 - Budapest, Hungary}

\email{izsakf@caesar.elte.hu}
\thanks{This work was completed in the ELTE Institutional Excellence Program
(1783-3/2018/FEKUTSRAT) supported by the Hungarian Ministry of Human Capacities.
}

\author{G\'abor Maros}
\address{Department of Applied Analysis and Computational Mathematics,
E\"otv\"os Lor\'and University, P\'azm\'any P. stny. 1C, 1117 - Budapest, Hungary}
\email{magabor42@gmail.com}
\thanks{The project has been supported by the European Union, co-financed by the European Social Fund (EFOP-3.6.3-VEKOP-16-2017-00002).}

\subjclass[2000]{Primary }




\begin{abstract}
Fractional-order elliptic problems are investigated in case of inhomogeneous Dirichlet
boundary data. The boundary integral form is proposed as a suitable mathematical model.
The corresponding theory is completed by sharpening the mapping properties of the
corresponding potential operators. Also a mild condition is provided to ensure the
existence of the classical solution of the boundary integral equation.
\end{abstract}

\maketitle



\section{Introduction
}

Analytic and numerical study of space-fractional diffusion problems became
an important research area  in the past twenty years. A large number of real-life
observations confirmed the presence of the of the so-called anomalous (or fractional)
diffusion.
The fractional Laplacian, which can be given in many equivalent ways in $\mathbb{R}^d$,
seems to be the most adequate differential operator for modeling this
phenomena.
At the same time, on a bounded domain, the incorporation of boundary conditions into
a true mathematical model is by far not easy. Even to deal with
\emph{homogeneous} Dirichlet boundary data is non-trivial \cite{baeumer18}: the most
meaningful \cite{izsak17} approach is given by the power of the Dirichlet-Laplacian.
A similar approach can be performed for the case of the homogeneous Neumann (no-flux)
boundary conditions.

At the same time, only a few attempts \cite{abatangelo17} and \cite{antil17}
were made to incorporate \emph{inhomogeneous} boundary conditions into a partial differential
equation, which is given on a bounded domain (corresponding to the real-life situation).

In any case, we should use fractional order differential operators, which are non-local.
At the same time, in real-life situations, we do not have any data outside of a physical
domain. This basic difficulty
motivates us to find an appropriate extension, which is formulated on $\mathbb{R}^d$.
Such an approach was successfully applied in \cite{szekeres15} but the corresponding
extension can not be used for arbitrary domains in $\mathbb{R}^d$.
Also, the simple choice in \cite{du13} and \cite{borthagaray17} can not
be applied to inhomogeneous boundary data.
Another difficulty in the analysis is that no generalization of the Green formula
is available for the fractional Laplacian.


It was pointed out in \cite{chang12} that the fractional order elliptic equations
with inhomogeneous boundary conditions can be succesfully analyzed in the framework of
 boundary integral equations. Our aim is to extend this result in the following sense:

\begin{itemize}
\item
the well-posedness is stated in two space dimensions,
\item
the mapping properties of the corresponding single layer operator
are generalized,
\item
a condition is given for the existence of a classical (pointwise)
solution.
\end{itemize}

An important motivation of this study is to prepare a numerical simulation in 2D, where the boundary integral form reduces the problem to compute one-dimensional integrals.

\subsection{Mathematical preliminaries}

The main problem in this study is the precise statement and the analysis of the elliptic boundary value problem
\begin{equation}\label{fbvp}
\begin{cases}
- (-\Delta)^\alpha u(\mathbf{x}) = f(\mathbf{x}) \quad \mathbf{x}\in\Omega\\
 u(\mathbf{x}) = g(\mathbf{x}) \quad \mathbf{x}\in\partial\Omega,
\end{cases}
\end{equation}
where $\Omega\subset\mathbb{R}^d$ is a bounded Lipschitz domain ($d=2,3$),
$\alpha\in (\frac{1}{2},1)$ and $f, g$ are given real functions.
At this stage, the differential operator $(-\Delta)^\alpha$ is not yet defined.
In any case, it should be linear, such that $u$ can be given as $u = u_1 + u_2$, where
\begin{equation}\label{fbvp1}
\begin{cases}
- (-\Delta)^\alpha u_1(\mathbf{x}) = f(\mathbf{x}) \quad \mathbf{x}\in\Omega\\
 u_1(\mathbf{x}) = 0 \quad \mathbf{x}\in\partial\Omega
\end{cases}
\end{equation}
and
\begin{equation}\label{fbvp2}
\begin{cases}
- (-\Delta)^\alpha u_2(\mathbf{x}) = 0 \quad \mathbf{x}\in\Omega\\
 u_2(\mathbf{x}) = g(\mathbf{x}) \quad \mathbf{x}\in\partial\Omega.
\end{cases}
\end{equation}
To deal with these problems, we recall that the fractional Dirichlet Laplacian $(-\Delta_\mathcal{D})^\alpha$ is defined as a power of
$$
-\Delta_{\mathcal{D}}: L_2(\Omega)\to L_2(\Omega), \; \textrm{Dom}(-\Delta_{\mathcal{D}}) = H_0^1(\Omega)\cap H^2(\Omega).
$$
This makes sense since
$(-\Delta_{\mathcal{D}})^{-1}: L_2(\Omega)\to L_2(\Omega)$ is compact, positive and self-adjoint.

Accordingly, $u_1$ in \eqref{fbvp1} can be given as $u_1 = -(-\Delta_\mathcal{D})^{-\alpha} f$. 
Therefore, we shall focus to rewrite the problem in \eqref{fbvp2} for $u_2$.

The fractional Laplacian on $\mathbb{R}^d$ has many equivalent definitions \cite{kwasnicki17}.
This operator can be defined pointwise as
\begin{equation}\label{pw_lap}
-(-\Delta_{\mathbb{R}^d})^\alpha u(\mathbf{x}) =
\lim_{r\to 0^+}
\frac{2^{2\alpha} \Gamma(\frac{d}{2}+\alpha)}{\pi^{\frac{d}{2}}\Gamma(-\alpha)}
\int_{B(\mathbf{0},r)^C} \frac{u(\mathbf{x}+\mathbf{z}) - u(\mathbf{x})}{|\mathbf{z}|^{d+2\alpha}}\:\mathrm{d}\mathbf{z},
\end{equation}
where $B(\mathbf{x},r)^C = \mathbb{R}^d\setminus B(\mathbf{x},r)$.
Accordingly, the weak fractional Laplacian
$(-\Delta_{\mathbb{R}^d})^\alpha u$ can be given as the function for which
$$
\int_{\mathbb{R}^d}
v (-\Delta_{\mathbb{R}^d})^\alpha u
=
\int_{\mathbb{R}^d}
u (-\Delta_{\mathbb{R}^d})^\alpha v
$$
is satisfied for all $v\in C_0^\infty(\mathbb{R}^d)$.

We will make use of the fundamental solution $\phi_\alpha$ of $(-\Delta_{\mathbb{R}^d})^\alpha$, which is given with
\[
\phi_\alpha = \mathcal{F}^{-1} \frac{1}{|\textrm{Id}|^{2\alpha}} =
2^{\alpha-\frac{n}{2}}
\frac{\Gamma\left(\frac{\alpha}{2}\right)}
{\Gamma\left(\frac{n-\alpha}{2}\right)}
|\textrm{Id}|^{2\alpha-d}.
\]

In the analysis, we use the notation $H^s(\Omega)\subset L_{2}(\Omega)$  for the classical Sobolev spaces with arbitrary positive indices. Recall that the corresponding norms on $\mathbb{R}^d$
can be defined by using the Fourier transform $\mathcal{F}$ as follows:
\begin{equation}\label{frac_with_F}
\|u\|_{H^s(\mathbb{R}^d)}^2 = \int_{\mathbb{R}^d} (1 + |\mathbf{r}|^2)^s |\mathcal{F}u|^2 (\mathbf{r}) \:\mathrm{d}\mathbf{r}.
\end{equation}
For stating the well-posedness, we need also
the Sobolev space
\[
\dot H^s(\mathbb{R}^d)= \left\{ u\in L_{2,\textrm{loc}}(\Omega):
\int_{\mathbb{R}^d} |\mathbf{r}|^{2s} |\mathcal{F}u|^2 (\mathbf{r}) \:\mathrm{d}\mathbf{r}<\infty\right\}
\]
with the corresponding norm.
If the underlying domain ($\Omega$, $\partial\Omega$ or $\mathbb{R}^d$) is obvious, simply the
notation $\|\cdot\|_s$ will be used for the corresponding norms.

An important tool in the analysis is the trace operator $\gamma$.
For a bounded Lipschitz domain $\Omega$,
\begin{equation}\label{cont_gamma}
\gamma: H^s(\Omega)\to H^{s-\frac{1}{2}}(\partial\Omega)
\end{equation}
is continuous for $s\in (\frac{1}{2}, \frac{3}{2})$, see \cite{mclean00}, Theorem 3.38.
Also, one can define its Banach adjoint as a continuous operator with
\begin{equation}\label{cont_gammacs}
\gamma^*:  H^{\frac{1}{2}-s}(\partial\Omega)\to H^{-s}(\Omega).
\end{equation}
We use the conventional notation $\langle \cdot, \cdot\rangle_{-\beta, \beta}$ for the duality pairing between $H^{-\beta}(\partial\Omega)$ and $H^{\beta}(\partial\Omega)$ with some positive exponent
$\beta$.

We also recall that the Bessel function
$J_0:\mathbb{R}^+\to\mathbb{R}$
of first kind is given with
\begin{equation}\label{bessel}
J_0 (s) = \frac{1}{2\pi} \int_{-\pi}^{\pi} e^{-i s\sin \tau}\:\mathrm{d}\tau.
\end{equation}
In the estimations, the relation $K_1 \lesssim K_2$ means that $K_1\le C K_2$
for some domain-dependent constant $C\in\mathbb{R}^+$.

\subsection{The main objective, comparison with earlier achievements}

In \cite{chang12}, the problem in \eqref{fbvp2} for $u_2$ was transformed to a boundary integral equation and the
following result was established.
\begin{theorem}\label{thm_chang}
For any bounded Lipschitz domain $\Omega\subset\mathbb{R}^d$ with $d\ge 3$
and $g\in H^{\alpha-\frac{1}{2}}(\partial\Omega)$ with $\alpha\in (\frac{1}{2}, 1)$, the problem
\begin{equation}\label{main_chang}
\begin{cases}
- (-\Delta)^\alpha \tilde u(\mathbf{x}) = 0 \quad \mathbf{x}\in(\partial\Omega)^C\\
 \tilde u(\mathbf{x}) = g(\mathbf{x}) \quad \mathbf{x}\in\partial\Omega\\
 |\tilde u(\mathbf{x})| \lesssim |\mathbf{x}|^{2\alpha-d} \quad \mathbf{x}\in B(\mathbf{0},1)^C
\end{cases}
\end{equation}
has a unique weak solution $\tilde u\in \dot H^\alpha(\mathbb{R}^d)$. 
\end{theorem}

Observe that \eqref{main_chang} with the definition $u_2 := \tilde u|_{\Omega}$
delivers a precise setting for the second problem in \eqref{fbvp2}.
Our aim is to sharpen this result and prove that the unique solution of
\eqref{main_chang}
\begin{itemize}
\item[(i)]
exists \emph{also in case of $d=2$} provided that $\alpha\in (\frac{1}{2}, \frac{3}{4})$,
\item[(ii)]
satisfies $- (-\Delta)^\alpha \tilde u(\mathbf{x}) = 0$ \emph{also pointwise} for any
$\mathbf{x}\in(\partial\Omega)^C$ under some weak condition.
\end{itemize}



\section{Main results}
\label{sec:main}
\subsection{Estimates for single layer potentials}
We first investigate the fractional version $R_\alpha$ of the classical Newton potential,
which is defined with
\begin{equation}\label{R_alpha_def}
R_\alpha (u) (\mathbf{x}) = \int_{\mathbb{R}^d} \phi_\alpha (\mathbf{x}-\mathbf{y}) u(\mathbf{y}) \:\mathrm{d}\mathbf{y},
\end{equation}
and also called the Riesz potential.

\begin{lem}\label{S_vol_lem}
Assume that $\alpha\in (0,1)$ for $d=3$ or $\alpha\in (\frac{1}{2}, \frac{3}{4})$ for
$d=2$. The mapping $R_\alpha$ in \eqref{R_alpha_def} defines then a continuous linear
map between
$H^{s-2\alpha}(\Omega)$ and $H^s(\Omega)$, i.e. we have
\begin{equation}\label{R_alpha_cont}
\|R_\alpha (u)\|_s \lesssim   \|u\|_{s-2\alpha}.
\end{equation}
\end{lem}
\begin{proof}
We first define an extension of $R_\alpha (u)$ with
\[
\widetilde R_\alpha (u) (\mathbf{x}) = \int_{\mathbb{R}^d} \lambda(|\mathbf{x}-\mathbf{y}|) \phi_\alpha (\mathbf{x}-\mathbf{y}) u(\mathbf{y}) \:\mathrm{d}\mathbf{y},
\]
where $\lambda\in C_0^\infty[0,\infty)$ with $\lambda|_{[0,2\textrm{diam}\:\Omega]} = 1$
and $0\le\lambda\le 1$.
Since $\widetilde R_\alpha (u)$ is an extension of $R_\alpha(u)$, we obviously have
\begin{equation}\label{triv_ext}
 \|\tilde R_\alpha (u)\|_s  \ge \|R_\alpha(u)\|_s.
\end{equation}
Since this is a convolution, we can give its Fourier transform as
\begin{equation}\label{FTR}
\begin{aligned}
&\mathcal{F} \left[ \widetilde R_\alpha (u) \right] (\mathbf{r})\\
&=
\mathcal{F} \left[ u \right] (\mathbf{r}) \cdot
\mathcal{F} \left[ \lambda(|\mathbf{z}|) \phi_\alpha (\mathbf{z}) \right] (\mathbf{r})=
\frac{C_\alpha}{4\pi} \int_{\mathbb{R}^d} e^{-i\langle \mathbf{r},\mathbf{z}\rangle}
\lambda(|\mathbf{z}|) |\mathbf{z}|^{2\alpha-d}\:\mathrm{d}\mathbf{z}.
\end{aligned}
\end{equation}
We estimate the second component, which, using polar coordinates in 3D,
can be rewritten as
\[
\begin{aligned}
&
\mathcal{F} \left[ \lambda(|\mathbf{z}|) \phi_\alpha (\mathbf{z}) \right] (\mathbf{r})
=
\frac{C_\alpha}{2} \int_0^\infty r^{2\alpha-1} \lambda(r)
\int_0^\pi e^{-ir|\mathbf{r}|\cos\theta} \sin\theta\:\mathrm{d}\theta\:\mathrm{d}r\\
&=
\frac{C_\alpha}{2}
\int_0^\infty r^{2\alpha-1} \lambda(r)
\frac{\sin(r|\mathbf{r}|)}{r|\mathbf{r}|}\:\mathrm{d}r.
\end{aligned}
\]

For $|\mathbf{r}|\ge 1$, we introduce $w=r|\mathbf{r}|$ and use that $\lambda$ is compactly supported
and $2\alpha-2\le 0$ to get
\begin{equation}\label{F_est_for_large}
\mathcal{F} \left[ \lambda(|\mathbf{z}|) \phi_\alpha (\mathbf{z}) \right] (\mathbf{r})
=
\frac{C_\alpha}{2 |\mathbf{r}|^{2\alpha}}
\int_0^\infty w^{2\alpha-1} \lambda\left(\frac{w}{|\mathbf{r}|}\right)\frac{\sin w}{w}
\:\mathrm{d}w
\lesssim
\frac{1}{|\mathbf{r}|^{2\alpha}}
\lesssim
\frac{1}{|1 + \mathbf{r}^2|^{\alpha}}.
\end{equation}

For $|\mathbf{r}|\le 1$, we only use again use that $\lambda$ is compactly supported
and $2\alpha-1> -1$ to have
\begin{equation}\label{F_est_for_small}
\mathcal{F} \left[ \lambda(|\mathbf{z}|) \phi_\alpha (\mathbf{z}) \right] (\mathbf{r})
=
\frac{C_\alpha}{2}
\int_0^\infty r^{2\alpha-1} \lambda(r)
\frac{\sin(r|\mathbf{r}|)}{r|\mathbf{r}|}\:\mathrm{d}r
\lesssim
1.
\end{equation}
In the 2D case, the polar transformation and the definition of $J_0$ in \eqref{bessel} gives
\begin{equation*}
\begin{aligned}
&\mathcal{F} \left[ \lambda(|\mathbf{z}|) \phi_\alpha (\mathbf{z}) \right] (\mathbf{r})
=
\frac{C_\alpha}{2} \int_0^\infty r^{2\alpha-1} \lambda(r)
\int_0^{2\pi} e^{-ir|\mathbf{r}|\cos\theta}\:\mathrm{d}\theta\:\mathrm{d}r\\
&=
C_\alpha\pi
\int_0^\infty r^{2\alpha-1} \lambda(r) J_0(r|\mathbf{r}|)\:\mathrm{d}r.
\end{aligned}
\end{equation*}

For $|\mathbf{r}|\ge 1$ with $w=r|\mathbf{r}|$, we apply the estimate
$J_0(w) \sqrt{w} \lesssim 1$ (see 9.2.1 in \cite{abramowitz64}), which implies that
\begin{equation}\label{F_est_for_large2}
\begin{aligned}
&\mathcal{F} \left[ \lambda(|\mathbf{z}|) \phi_\alpha (\mathbf{z}) \right] (\mathbf{r})
=
\frac{C_\alpha}{2 |\mathbf{r}|^{2\alpha}}
\int_0^\infty w^{2\alpha-1} \lambda\left(\frac{w}{|\mathbf{r}|}\right) J_0(w)
\:\mathrm{d}w
\lesssim
\int_0^\infty w^{2\alpha-\frac{3}{2}} \lambda\left(\frac{w}{|\mathbf{r}|}\right)
\:\mathrm{d}w\\
&\lesssim
\frac{1}{|1 + \mathbf{r}^2|^{\alpha}}.
\end{aligned}
\end{equation}
provided that $2\alpha -\frac{3}{2} \in (-1,0]$, i.e. $\alpha \in (\frac{1}{4}, \frac{3}{4}]$.\\
For $|\mathbf{r}|\le 1$, we again only use $2\alpha-1> -1$ and the boundedness of the remaining components to have
\begin{equation}\label{F_est_for_small2}
\mathcal{F} \left[ \lambda(|\mathbf{z}|) \phi_\alpha (\mathbf{z}) \right] (\mathbf{r})
=
C_\alpha\pi
\int_0^\infty r^{2\alpha-1} \lambda(r) J_0(r|\mathbf{r}|)\:\mathrm{d}r
\lesssim
1.
\end{equation}
Using \eqref{triv_ext} together with \eqref{frac_with_F} and the estimates \eqref{F_est_for_large},
 \eqref{F_est_for_small}, \eqref{F_est_for_large2} and \eqref{F_est_for_small2} in
\eqref{FTR}, we finally obtain that
\[
\begin{aligned}
&
\|R_\alpha (u)\|^2_s  \le \|\tilde R_\alpha (u)\|^2_s
=
\int_{\mathbb{R}^d}
(1+|\mathbf{r}|^2)^s \left|\mathcal{F} \left[ \tilde R_\alpha (u) \right] (\mathbf{r})\right|^2 \:\mathrm{d}\mathbf{r}\\
&=
\int_{|\mathbf{r}|\le 1}
|\mathcal{F} \left[ u \right]|^2 (\mathbf{r})
(1+|\mathbf{r}|^2)^s
\left|\mathcal{F} \left[ \lambda(|\mathbf{z}|) \phi_\alpha (\mathbf{z}) \right]\right|^2 (\mathbf{r}) \:\mathrm{d}\mathbf{r}\\
&\quad +
\int_{|\mathbf{r}|\ge 1}
|\mathcal{F} \left[ u \right]|^2 (\mathbf{r})
(1+|\mathbf{r}|^2)^s
\left|\mathcal{F} \left[ \lambda(|\mathbf{z}|) \phi_\alpha (\mathbf{z}) \right]\right|^2
(\mathbf{r}) \:\mathrm{d}\mathbf{r}\\
&\le
\int_{|\mathbf{r}|\le 1}
|\mathcal{F} \left[ u \right]|^2 (\mathbf{r})
(1+\mathbf{r}^2)^s \tilde C_R^2 (\mathbf{r}) \:\mathrm{d}\mathbf{r}
+
\int_{|\mathbf{r}|\ge 1}
|\mathcal{F} \left[ u \right]|^2 (\mathbf{r})
(1+|\mathbf{r}|^2)^s \frac{2 C_R}{|1 + \mathbf{r}^2|^{2\alpha}} \:\mathrm{d}\mathbf{r}\\
&\le
C_R
\int_{\mathbb{R}^d}
|\mathcal{F} \left[ u \right]|^2 (\mathbf{r})
(1+|\mathbf{r}|^2)^{s - 2\alpha}  \:\mathrm{d}\mathbf{r}
=
\|u\|_{s - 2\alpha}^2,
\end{aligned}
\]
as stated in the lemma.
\end{proof}

We need, however, the surface potential corresponding to $R_\alpha$ on
$\partial\Omega$, which is given for any $\mathbf{x}\in\mathbb{R}^d$ with
\begin{equation}\label{S_alpha_def}
S_\alpha (u) (\mathbf{x}) = \int_{\partial\Omega} \phi_\alpha (\mathbf{x}-\mathbf{y}) u(\mathbf{y}) \:\mathrm{d}\mathbf{y}.
\end{equation}
In precise terms, we state the following generalization of (4.1) in \cite{chang12}.

\begin{theorem}\label{S_thm}
For any indices $s,\alpha$ satisfying the assumptions in Lemma \ref{S_vol_lem} and
$2\alpha - s \in (\frac{1}{2}, \frac{3}{2})$, the mapping
$S_\alpha$ defines a continuous linear operator between
$H^{s-2\alpha+\frac{1}{2}}(\partial\Omega)$ and $H^s(\Omega)$, i.e. for all
 $\psi\in H^{s-2\alpha+\frac{1}{2}}(\partial\Omega)$, we have
\begin{equation}\label{def_S}
\|S_\alpha (\psi)\|_s \lesssim  \|\psi\|_{s-2\alpha+\frac{1}{2},\partial\Omega}.
\end{equation}
\end{theorem}
\begin{proof}
We first use the definitions in \eqref{R_alpha_def}, \eqref{S_alpha_def}
and the adjoint trace operator $\gamma^*$ in \eqref{cont_gammacs} to rewrite
$S_\alpha(\psi)$ as
\[
S_\alpha(\psi)(\mathbf{x})
=
\int_{\partial\Omega} \phi_\alpha (\mathbf{x}-\mathbf{y}) \psi(\mathbf{y}) \:\mathrm{d}\mathbf{y}
=
 \int_{\Omega} \phi_\alpha (\mathbf{x}-\mathbf{y}) (\gamma^*\psi)(\mathbf{y}) \:\mathrm{d}\mathbf{y}
=
R_\alpha\gamma^*(\psi)(\mathbf{x}).
\]
The smoothness of $\phi_\alpha (\mathbf{x}-\cdot)$ also implies that
$S_\alpha(\psi)\in H^s(\Omega)$. Accordingly, using also \eqref{R_alpha_cont},
\eqref{cont_gamma} and \eqref{cont_gammacs}, we have for any $\phi\in H^{-s}(\Omega)$
and $\psi\in H^{s-2\alpha+\frac{1}{2}}(\partial\Omega)$ that
\[
\begin{aligned}
&\left|\langle \phi, S_\alpha(\psi) \rangle_{-s,s}\right|
=
\left|\langle \phi, R_\alpha\gamma^*(\psi) \rangle_{-s,s}\right|
\le
\|R_\alpha \gamma^*\psi\|_{s}  \|\phi\|_{-s}
\lesssim
\|\gamma^*\psi\|_{s-2\alpha}  \|\phi\|_{-s}\\
&\lesssim
\|\psi\|_{s-2\alpha+\frac{1}{2},\partial\Omega}  \|\phi\|_{-s}.
\end{aligned}
\]
Therefore, we arrive at the estimate
$$
\|S_\alpha(\psi)\|_s \lesssim  \|\psi\|_{s-2\alpha+\frac{1}{2},\partial\Omega},
$$
which completes the proof.
\end{proof}

\begin{theorem}\label{coerc_thm}
The map $\gamma S_{\alpha}$ is a coercive operator between $H^{1/2-\alpha}(\partial\Omega)$ and $H^{\alpha-1/2}(\partial\Omega)$ in the sense that
\begin{equation}\label{form_coerc}
\langle u, \gamma S_{\alpha}u \rangle_{\frac{1}{2}-\alpha, \alpha-\frac{1}{2}}
=
\int_{\partial\Omega}(S_{\alpha}u)(\mathbf{x})u(\mathbf{x})\:\mathrm{d}\mathbf{x}
\gtrsim
\|u\|_{1/2-\alpha,\partial\Omega}^2.
\end{equation}
\end{theorem}
\begin{proof}
We first recall that according to the proof of Theorem 4.1. in \cite{chang12}, the left hand side of \eqref{form_coerc} can be given as
\begin{equation}\label{main_from_chang}
\langle u, \gamma S_{\alpha}u \rangle_{\frac{1}{2}-\alpha, \alpha-\frac{1}{2}}
=
\int_{\mathbb{R}^d}|\mathbf{r}|^{4-2\alpha}|\mathcal{F}(S_{1}u)|^2(\mathbf{r})\:\mathrm{d}\mathbf{r}.
\end{equation}
In concrete terms, see (5.2) in \cite{chang12}.  Also, according to Theorem 4.1. in \cite{chang12}, for any $v\in H^1(\mathbb{R}^d)$ and $u\in H^{1/2-\alpha}(\partial\Omega)$, we can rewrite $\langle u, \gamma v \rangle_{\frac{1}{2}-\alpha, \alpha-\frac{1}{2}}$ as
$$
\langle u, \gamma v \rangle_{\frac{1}{2}-\alpha, \alpha-\frac{1}{2}}
=\int_{\mathbb{R}^d}\nabla S_1u(\mathbf{x})\nabla v(\mathbf{x})\:\mathrm{d}\mathbf{x}.
$$
The basic step for proving our statement is to rewrite the fractional order norm as follows:
\begin{equation}\label{eq1}
\|u\|_{1/2-\alpha,\partial\Omega}=\sup_{\substack{\phi\in H^{\alpha-1/2}(\partial\Omega) \\ \|\phi\|_{\alpha-1/2}=1}}|\langle u, \phi\rangle_{\frac{1}{2}-\alpha, \alpha-\frac{1}{2}}|=\sup_{\substack{\phi\in H^{1/2}(\partial\Omega) \\ \|\phi\|_{\alpha-1/2}=1}}|\langle u, \phi\rangle_{\frac{1}{2}-\alpha, \alpha-\frac{1}{2}}|,
\end{equation}
where in the second equality, we have used the density of $H^{1/2}(\partial\Omega)$ in $H^{\alpha-1/2}(\partial\Omega)$.\\
Let $\varepsilon:H^{s}(\partial\Omega)\rightarrow H^{s+1/2}(\mathbb{R}^d)$ denote the right inverse of the trace operator $\gamma$, which is continuous for $s\in(0,1)$, see \cite{mclean00}, Theorem 3.37. Using \eqref{eq1}, converting everything to the Fourier space,  applying the Cauchy--Schwarz inequality, the
formula in \eqref{frac_with_F} with the continuity of $\varepsilon$, we finally have

\begin{align*}
&\|u\|_{1/2-\alpha,\partial\Omega}=\sup_{\substack{\phi\in H^{1/2}(\partial\Omega) \\ \|\phi\|_{\alpha-1/2}=1}}\int_{\mathbb{R}^d}\nabla S_1 u(\mathbf{x}) \nabla \varepsilon \phi(\mathbf{x}\:\mathrm{d}\mathbf{x})\\
&=\sup_{\substack{\phi\in H^{1/2}(\partial\Omega)\\ \|\phi\|_{\alpha-1/2}=1}}\int_{\mathbb{R}^d}|\mathbf{r}|^2 \mathcal{F}(S_1 u)(\mathbf{r}) \mathcal{F}(\varepsilon \phi)(\mathbf{r})\:\mathrm{d}\mathbf{r}\\
&=
\sup_{\substack{\phi\in H^{1/2}(\partial\Omega)\\ \|\phi\|_{\alpha-1/2}=1}}\int_{\mathbb{R}^d}\frac{|\mathbf{r}|^2}{(1+|\mathbf{r}|)^{\alpha}} \mathcal{F}(S_1 u)(\mathbf{r})\cdot (1+|\mathbf{r}|)^{\alpha}\mathcal{F}(\varepsilon \phi)(\mathbf{r})\:\mathrm{d}\mathbf{r}\\
&\leq
\sup_{\substack{\phi\in H^{1/2}(\partial\Omega)\\ \|\phi\|_{\alpha-1/2}=1}}
\left[\int_{\mathbb{R}^d}\frac{|\mathbf{r}|^4}{(1+|\mathbf{r}|)^{2\alpha}} |\mathcal{F}(S_1 u)|^2(\mathbf{r})\:\mathrm{d}\mathbf{r}\right]^{1/2}
\left[\int_{\mathbb{R}^d}(1+|\mathbf{r}|)^{2\alpha}\mathcal{F}(\varepsilon \phi)(\mathbf{r})(\mathbf{r})\:\mathrm{d}\mathbf{r}\right]^{1/2}\\
&\leq
\sup_{\substack{\phi\in H^{1/2}(\partial\Omega) \\ \|\phi\|_{\alpha-1/2}=1}} \left[\int_{\mathbb{R}^d}|\mathbf{r}|^{4-2\alpha} |\mathcal{F}(S_1 u)|^2(\mathbf{r})\:\mathrm{d}\mathbf{r}\right]^{1/2}
\|\varepsilon \phi\|_{\alpha,\mathbb{R}^d}\\
&\lesssim
\sup_{\substack{\phi\in H^{1/2}(\partial\Omega) \\ \|\phi\|_{\alpha-1/2}=1}}
\left[\int_{\mathbb{R}^d}|\mathbf{r}|^{4-2\alpha} |\mathcal{F}(S_1 u)|^2(\mathbf{r})\:\mathrm{d}\mathbf{r}\right]^{1/2}
\| \phi\|_{\alpha-1/2,\partial\Omega}\\
& = \left[\int_{\mathbb{R}^d}|\mathbf{r}|^{4-2\alpha} |\mathcal{F}(S_1 u)|^2(\mathbf{r})\:\mathrm{d}\mathbf{r}\right]^{1/2}.
\end{align*}

Therefore, using \eqref{main_from_chang}, we get
$$
\langle u, \gamma S_{\alpha}u \rangle_{\frac{1}{2}-\alpha, \alpha-\frac{1}{2}}
=
\int_{\mathbb{R}^d}|\mathbf{r}|^{4-2\alpha}|\mathcal{F}(S_{1}u)|^2(\mathbf{r})\:\mathrm{d}\mathbf{r}\gtrsim \|u\|_{1/2-\alpha,\partial\Omega}^2,
$$
as stated in the theorem.
\end{proof}

We are ready now to prove the main statement of the article.
\begin{theorem}\label{main_thm}
Assume that $\alpha\in(\frac{1}{2}, 1)$ for $d=3$ or $\alpha\in(\frac{1}{2}, \frac{3}{4}]$ for $d=2$.
Then for any $g\in H^{\alpha-\frac{1}{2}}(\partial\Omega)$, there
is a unique function $G\in H^{\frac{1}{2}-\alpha}(\partial\Omega)$ such that
$u = S_\alpha (G) \in \dot H^\alpha(\mathbb{R}^d)$ solves the problem in \eqref{main_chang}.\\
If, we have additionally $G\in L_1(\partial\Omega)$, then the pointwise equality
$-(-\Delta)^\alpha \tilde u = 0$ in $(\partial\Omega)^C$ is also satisfied.
\end{theorem}
\begin{proof}
Taking the special case $\alpha = s$ in Theorem \ref{S_thm}, we have that
$$
\gamma S_\alpha: H^{\frac{1}{2}-\alpha}(\partial\Omega)
\to H^{\alpha-\frac{1}{2}}(\partial\Omega).
$$
is continuous. Also, according to Theorem \ref{coerc_thm}, $\gamma S_\alpha$ is coercive.
Consequently, this constitutes a bijection between $H^{\frac{1}{2}-\alpha}(\partial\Omega)$
and $H^{\alpha-\frac{1}{2}}(\partial\Omega)$, so that for any given $g\in H^{\alpha-\frac{1}{2}}(\partial\Omega)$, there
is a unique function $G\in H^{\frac{1}{2}-\alpha}(\partial\Omega)$ such that
$$
\gamma S_\alpha (G) = g.
$$
Also, as pointed out in \cite{chang12}, $S_\alpha (G)\in \dot H^\alpha(\mathbb{R}^d)$
and
$$
-(-\Delta_{\mathbb{R}^d})^\alpha  S_\alpha (G) = 0\quad\textrm{on}\; \Omega^C
$$
in weak sense, so that \eqref{main_chang} with the conditions in the present
theorem has a unique solution.


To prove the pointwise identity $-(-\Delta)^\alpha \tilde u = 0$ in $(\partial\Omega)^C$, we use the definition in \eqref{pw_lap} so that
\begin{equation*}
-(-\Delta_{\mathbb{R}^d})^\alpha S_\alpha(G)(\mathbf{x}) =
\lim_{r\to 0^+}
\frac{2^{2\alpha} \Gamma(\frac{d}{2}+\alpha)}{\pi^{\frac{d}{2}}\Gamma(-\alpha)}
\int_{B(\mathbf{0},r)^C} \frac{S_\alpha(G)(\mathbf{x}+\mathbf{z}) - S_\alpha(G)(\mathbf{x})}
{|\mathbf{z}|^{d+2\alpha}}\:\mathrm{d}\mathbf{z}.
\end{equation*}
Therefore, by the definition in \eqref{S_alpha_def}, we have to verify that
\begin{equation}\label{pw_lap2}
\lim_{r\to 0^+}
\int_{B(\mathbf{0},r)^C}
\frac{1}{|\mathbf{z}|^{d+2\alpha}}
\int_{\partial\Omega} (\phi_\alpha(\mathbf{x}+\mathbf{z}-\mathbf{y})- \phi_\alpha(\mathbf{x}-\mathbf{y})) G(\mathbf{y}) \:\mathrm{d}\mathbf{y}\:\mathrm{d}\mathbf{z}
=0.
\end{equation}
We may assume that here $r\le r^*:= \frac{1}{2}d(\mathbf{x},\Omega)$ and note that
$|\mathbf{z}|\lesssim |\mathbf{x}-\mathbf{y}+\mathbf{z}|$ on $B(\mathbf{0},r)^C\setminus B(\mathbf{y}-\mathbf{x}, r^*)$, which,
along with the definition of $\phi_\alpha$, imply that
$$
\begin{aligned}
&\int_{B(\mathbf{0},r)^C}
\frac{\phi_\alpha(\mathbf{x}+\mathbf{z}-\mathbf{y})}{|\mathbf{z}|^{d+2\alpha}}\:\mathrm{d}\mathbf{z}
=
\int_{B(\mathbf{0},r)^C\setminus B(\mathbf{y}-\mathbf{x},r^*)}
\frac{\phi_\alpha(\mathbf{x}-\mathbf{y}+\mathbf{z})}{|\mathbf{z}|^{d+2\alpha}}\:\mathrm{d}\mathbf{z}\\
&\quad+
\int_{B(\mathbf{y}-\mathbf{x},r^*)}
\frac{\phi_\alpha(\mathbf{x}-\mathbf{y}+\mathbf{z})}{|\mathbf{z}|^{d+2\alpha}}\:\mathrm{d}\mathbf{z}\\
&\lesssim
\int_{B(\mathbf{0},r)^C}
\frac{|\mathbf{z}|^{2\alpha-d}}{|\mathbf{z}|^{d+2\alpha}}\:\mathrm{d}\mathbf{z}
+
\int_{B(\mathbf{0},r^*)}
\frac{\phi_\alpha(\mathbf{z}_0)}{|\mathbf{z}_0-(\mathbf{x}-\mathbf{y})|^{d+2\alpha}}\:\mathrm{d}\mathbf{z}_0\\
&\lesssim
\int_r^\infty s^{-2d}s^{d-1}\:\mathrm{d}s
+
\int_0^{r^*} s^{2\alpha-d}s^{d-1}\:\mathrm{d}s
\lesssim K < \infty,
\end{aligned}
$$
which can be used to get the following inequality:
\begin{equation}\label{fin_int}
\begin{aligned}
&\int_{B(\mathbf{0},r)^C}
\frac{|\phi_\alpha(\mathbf{x}+\mathbf{z}-\mathbf{y})- \phi_\alpha(\mathbf{x}-\mathbf{y})|}{|\mathbf{z}|^{d+2\alpha}}
   \:\mathrm{d}\mathbf{z}\:\mathrm{d}\mathbf{y}
\lesssim
K+ \int_r^\infty \phi_\alpha(r^*)
\frac{s^{d-1}}{s^{d+2\alpha}}\:\mathrm{d}s \\
&\lesssim
\tilde K
< \infty,
\end{aligned}
\end{equation}
where $K$ does not depend on $\mathbf{y}$.
Therefore, we also have the following estimate:
$$
\int_{\partial\Omega}
|G(\mathbf{y})|
\int_{B(\mathbf{0},r)^C}
\frac{|\phi_\alpha(\mathbf{x}+\mathbf{z}-\mathbf{y})- \phi_\alpha(\mathbf{x}-\mathbf{y})|}{|\mathbf{z}|^{d+2\alpha}}
   \:\mathrm{d}\mathbf{z}\:\mathrm{d}\mathbf{y}
\lesssim
 \|G\|_{L_1(\Omega)} \tilde K
< \infty,
$$
such that we can apply Tonelli's theorem in \eqref{pw_lap2} together with
\eqref{fin_int} to obtain
\begin{equation}\label{ch_order}
\begin{aligned}
&
\int_{B(\mathbf{0},r)^C}
\frac{1}{|\mathbf{z}|^{d+2\alpha}}
\int_{\partial\Omega} (\phi_\alpha(\mathbf{x}+\mathbf{z}-\mathbf{y})- \phi_\alpha(\mathbf{x}-\mathbf{y})) G(\mathbf{y}) \:\mathrm{d}\mathbf{y}\:\mathrm{d}\mathbf{z} \\
&=
\int_{\partial\Omega}
G(\mathbf{y})
\int_{B(\mathbf{0},r)^C}
\frac{\phi_\alpha(\mathbf{x}+\mathbf{z}-\mathbf{y})- \phi_\alpha(\mathbf{x}-\mathbf{y})}{|\mathbf{z}|^{d+2\alpha}}\:\mathrm{d}\mathbf{z}\:\mathrm{d}\mathbf{y}.
\end{aligned}
\end{equation}
Using the equality  $(-\Delta)^\alpha \phi_\alpha (\mathbf{x}-\mathbf{y}) = 0$, we have that here
$$
\lim_{r\to 0}
G(\mathbf{y})
\int_{B(\mathbf{0},r)^C}
\frac{\phi_\alpha(\mathbf{x}+\mathbf{z}-\mathbf{y})- \phi_\alpha(\mathbf{x}-\mathbf{y})}{|\mathbf{z}|^{d+2\alpha}}\:\mathrm{d}\mathbf{y}\:\mathrm{d}\mathbf{z}
= 0
$$
and by \eqref{fin_int}, we also obtain that
$$
\left| G(\mathbf{y})
\int_{B(\mathbf{0},r)^C}
\frac{\phi_\alpha(\mathbf{x}+\mathbf{z}-\mathbf{y})- \phi_\alpha(\mathbf{x}-\mathbf{y})}{|\mathbf{z}|^{d+2\alpha}}\:\mathrm{d}\mathbf{z}\right|
\lesssim
 \tilde K_1 | G(\mathbf{y})|
$$
such that we can use the dominated convergence theorem in \eqref{ch_order}
to obtain the desired equality in \eqref{pw_lap2}.
\end{proof}




\section*{Acknowledgments}
The project has been supported by the European Union, co-financed by the European Social Fund (EFOP-3.6.3-VEKOP-16-2017-00002).
This work was completed in the ELTE Institutional Excellence Program
(1783-3/2018/FEKUTSRAT) supported by the Hungarian Ministry of Human Capacities.


\bibliography{ferenc_bib}
\bibliographystyle{abbrv}

\end{document}